\documentclass[12pt]{article}
\usepackage{graphicx}
\usepackage{amsmath,amsthm,amssymb,enumerate}%, esint}
\usepackage{euscript,mathrsfs}
\usepackage[left=2cm,right=2cm,top=3.5cm,bottom=3.5cm]{geometry}
\usepackage{color}
\catcode`\@=11 \@addtoreset{equation}{section}

\catcode`\@=12

\allowdisplaybreaks

\newtheorem{Theorem}{Theorem}[section]
\newtheorem{Proposition}[Theorem]{Proposition}
\newtheorem{Lemma}[Theorem]{Lemma}
\newtheorem{Corollary}[Theorem]{Corollary}

\theoremstyle{definition}
\newtheorem{Definition}[Theorem]{Definition}

\newtheorem{Remark}[Theorem]{Remark}

\newcommand{\bTheorem}[1]{
\begin{Theorem} \label{T#1} }
\newcommand{\eT}{\end{Theorem}}

\newcommand{\bProposition}[1]{
\begin{Proposition} \label{P#1}}
\newcommand{\eP}{\end{Proposition}}

\newcommand{\bLemma}[1]{
\begin{Lemma} \label{L#1} }
\newcommand{\eL}{\end{Lemma}}

\newcommand{\bCorollary}[1]{
\begin{Corollary} \label{C#1} }
\newcommand{\eC}{\end{Corollary}}

\newcommand{\bRemark}[1]{
\begin{Remark} \label{R#1} }
\newcommand{\eR}{\end{Remark}}
\newcommand{\E}{\mathcal{E}}
\newcommand{\bDefinition}[1]{
\begin{Definition} \label{D#1} }
\newcommand{\eD}{\end{Definition}}

\newcommand{\Q}{\mathbb{T}^N}

\newcommand{\bfphi}{\boldsymbol{\varphi}}

\newcommand{\bFormula}[1]{
\begin{equation} \label{#1}}
\newcommand{\eF}{\end{equation}}

\newcommand{\Ov}[1]{\overline{#1}}

\newcommand{\vr}{\varrho}

\newcommand{\tvr}{\tilde \vr}

\newcommand{\vm}{\vc{m}}

\newcommand{\vc}[1]{{\bf #1}}

\newcommand{\Div}{{\rm div}_x}
\newcommand{\Grad}{\nabla_x}

\newcommand{\dx}{\,{\rm d} {x}}

\newcommand{\dt}{\,{\rm d} t }

\newcommand{\vU}{\vc{U}}

\newcommand{\intO}[1]{\int_{\Q} #1 \ \dx}

\newcommand{\D}{{\rm d}}
\newcommand{\ep}{\varepsilon}

\definecolor{Cgrey}{rgb}{0.85,0.85,0.85}
\definecolor{Cblue}{rgb}{0.50,0.85,0.85}
\definecolor{Cred}{rgb}{1,0,0}
\definecolor{fancy}{rgb}{0.10,0.85,0.10}

\newcommand\Cbox[2]{%
    \newbox\contentbox%
    \newbox\bkgdbox%
    \setbox\contentbox\hbox to \hsize{%
        \vtop{
            \kern\columnsep
            \hbox to \hsize{%
                \kern\columnsep%
                \advance\hsize by -2\columnsep%
                \setlength{\textwidth}{\hsize}%
                \vbox{
                    \parskip=\baselineskip
                    \parindent=0bp
                    #2
                }%
                \kern\columnsep%
            }%
            \kern\columnsep%
        }%
    }%
    \setbox\bkgdbox\vbox{
        \color{#1}
        \hrule width  \wd\contentbox %
               height \ht\contentbox %
               depth  \dp\contentbox
        \color{black}
    }%
    \wd\bkgdbox=0bp%
    \vbox{\hbox to \hsize{\box\bkgdbox\box\contentbox}}%
    \vskip\baselineskip%
}

\date{}

\begin{document}

%%%%%%%%%%%%%%%%%%%%%%%%%%%%%%%%

\title{On uniqueness of dissipative solutions to the isentropic Euler system}

\author{Eduard Feireisl
\thanks{The research of E.F.~leading to these results has received funding from the
Czech Sciences Foundation (GA\v CR), Grant Agreement
18--05974S. The Institute of Mathematics of the Academy of Sciences of
the Czech Republic is supported by RVO:67985840.} \and Shyam Sundar Ghoshal
\thanks{S.G. thanks Inspire faculty-research grant 
 DST/INSPIRE/04/2016/000237} 
\and Animesh Jana
}

\date{\today}

\maketitle

\bigskip

\centerline{Institute of Mathematics of the Academy of Sciences of the Czech Republic}

\centerline{\v Zitn\' a 25, CZ-115 67 Praha 1, Czech Republic}

\centerline{and}

\centerline{Institute of Mathematics, TU Berlin}

\centerline{Strasse des 17.Juni, Berlin, Germany}

\bigskip

\centerline{Centre for Applicable Mathematics, Tata Institute of Fundamental Research}

\centerline{Post Bag No 6503, Sharadanagar, Bangalore - 560065, India}

\begin{abstract}

The dissipative solutions can be seen as a convenient generalization of the concept of weak solution to the isentropic Euler system. 
They can be seen as expectations of the Young measures associated to a suitable measure--valued solution of the problem. We show 
that dissipative solutions coincide with weak solutions starting from the same initial data on condition that:
{\bf (i)} the weak solution enjoys certain Besov regularity; {\bf (ii)} the symmetric velocity gradient of the weak solution satisfies 
a one--sided Lipschitz bound.

\end{abstract}

{\bf Keywords:} Euler system, compressible fluid, weak solution, dissipative solution, uniqueness.

%\tableofcontents

\section{Introduction}
\label{I}

An iconic example of a system of partial differential equations arising in continuum fluid mechanics is the (barotropic) 
\emph{Euler system}:
\begin{equation} \label{I1}
\partial_t \vr + \Div \vm = 0,\ \partial_t \vm + \Div \left( \frac{\vm \otimes \vm}{\vr} \right) + \Grad p(\vr) = 0. 
\end{equation}
The unknown functions $\vr = \vr(t,x)$ and $\vm = \vm(t,x)$ are the density and the momentum respectively of a compressible viscous fluid, $p = p(\vr)$ is the barotropic pressure. The fluid occupies a domain $\Omega \subset R^N$, $N=1,2,3$ and its initial state at 
a reference time $t = 0$ is prescribed:
\begin{equation} \label{I2} 
\vr(0, \cdot) = \vr_0,\ \vm(0, \cdot) = \vm_0. 
\end{equation}
To avoid problems connected with the presence of kinematic boundary, we impose the space--periodic boundary conditions, meaning $\Omega$ can be identified with a flat torus:
\begin{equation} \label{I3}
\Omega = \Q \equiv \left( [-1,1]|_{\{ -1, 1 \}} \right)^N,\ N=1,2,3.
\end{equation}
Note that this is not a major restrictions as solutions are expected to obey the finite speed of propagation, in particular, the 
initial data with $\vm_0$ compactly supported give rise to solutions enjoying the same property.
Finally, we suppose the isentropic pressure density state equation
\[
p(\vr) = a \vr^{\gamma}, \ a > 0,\ \gamma > 1.
\]

Well--posedness of the problem \eqref{I1}--\eqref{I3} is a delicate issue. Summarizing the most recent state--of--art 
we can assert:
\begin{itemize}
\item
Euler system \eqref{I1}--\eqref{I3} is well--posed (there is a unique solution) 
locally in time for sufficiently regular initial data, see e.g. the monographs Benzoni-Gavage and Serre \cite{BenSer} or 
Majda \cite{Majd};
\item regular solutions blow up (develop a shock wave type singularity) in a finite time for fairly general class of initial data, see e.g. Smoller \cite{SMO};    
\item weak solutions exist for any initial data globally in time, see DiPerna \cite{DiP3}, Lions, Perthame, and Souganidis 
\cite{LPS} for $N=1$, and Chiodaroli \cite{Chiod} and \cite{Fei2016} for $N=2,3$; the problem is ill--posed in the multidimenional case -- there exist infinitely many global--in--time weak solutions for any (regular) initial data. 
\end{itemize}

The existence of infinitely many weak solutions in the multidimensional case has been revealed through the method of convex integration adapted to problems in fluid mechanics by DeLellis and Sz\' ekelyhidi \cite{DelSze3}. Many of these solutions apparently violate 
the energy balance (inequality) 
\begin{equation} \label{I4}
\partial_t e + \Div \left[ \Big( e + p \Big) \frac{\vm}{\vr} \right]  \leq 0,\ \mbox{with}\ e \equiv \frac{1}{2} \frac{|\vm|^2}{\vr} 
+ P(\vr), \ P(\vr) = \frac{a}{\gamma - 1} \vr^\gamma,
\end{equation}
or even its simple integrated form
\begin{equation} \label{I5} 
\frac{{\rm d}}{{\rm d}t} E(t)  \leq 0,\ E(t) = \intO{ e(t, \cdot) },\  E(0) =  \intO{ \left[ \frac{1}{2} \frac{|\vm_0|^2}{\vr_0} 
+ P(\vr_0) \right] }.
\end{equation}
The weak solution satisfying either of these conditions are termed \emph{admissible}. 

Unfortunately, as shown by Chiodaroli, DeLellis, Kreml \cite{ChiDelKre}, Chiodaroli, Kreml \cite{ChiKre}, or, recently, by 
Chiodaroli et al. \cite{ChKrMaSwI}, the Euler system remains essentially ill--posed even in the class of admissible weak solutions 
starting from regular initial data and satisfying the more restrictive form of the energy balance \eqref{I4}. 

Despite all the ill--posedness results, the admissible weak solutions enjoy the \emph{weak--strong uniqueness} principle: 
a weak solution coincides with the strong solution starting from the same initial data on the life--span of the latter, see
Dafermos \cite{Daf4}. This property is quite robust and has been extended to the class of more general measure--valued solutions
to the Euler system by Gwiazda, \'Swierczewska--Gwiazda, Wiedemann \cite{GSWW}. The key point is validity of the so--called 
relative energy inequality, see e.g. \cite[Proposition 3.1]{FeKr2014}:
\begin{equation} \label{I6}
\begin{split}
\int_{\Q} &\E \left( \vr, \vm \Big| r, \vc{U} \right) (\tau, \cdot) \dx \leq 
\intO{ \E \left( \vr_0, \vm_0 \Big| r_0, \vc{U}_0 \right) } \\
&+ \int_0^\tau \intO{\left[ (\vr \vU - \vm) \partial_t \vU + 
\vm \cdot \Grad \vU \cdot \Big( \vU - \frac{\vm}{\vr} \Big) +
 \Big( (p(r) - p(\vr) \Big) \Div \vU  \right] }\dt
\\
&+\int_0^\tau \intO{ \left[ (r - \vr) \partial_t P'(r) + (r \vc{U} - \vm ) \cdot \Grad P'(r) \right] }\dt, 
\end{split}
\end{equation} 
where 
\begin{equation} \label{I7}
\E \left( \vr, \vm \ \Big| \ r, \vU \right) 
\equiv \frac{1}{2} \vr \left| \frac{\vm}{\vr} - \vU \right|^2 + P(\vr) - P'(r) (\vr - r) - P(r),
\end{equation}
and $r > 0$, $\vU$ are arbitrary Lipschitz continuous ``test functions'', $r_0 \equiv r(0, \cdot)$, 
$\vU_0 \equiv \vU(0, \cdot)$.

In the light of the above mentioned results, it seems of interest to identify the largest class possible of admissible weak solutions
uniquely determined by their initial data. The iconic example for $N=1$ is the Riemann problem, where \eqref{I5} is sufficient to 
select the unique physically admissible weak solution. This is no longer true for $N=2,3$, where the certain Riemann data are known to give rise to infinitely admissible weak solutions, see e.g. \cite{ChiDelKre}. Uniqueness, however, is preserved in the multidimensional case for the rarefaction waves - weak solutions starting from certain Riemann data and becoming Lipschitz at any $t > 0$,  
see Chen and Chen \cite{CheChe}. It is the main goal of the present paper to show that uniqueness holds as soon as the weak solution $\vr = r$, $\vm = r \vc{U}$ satisfies the following conditions:
\begin{itemize}
\item {\it No vacuum.} The density is bounded below (and above) uniformly away from zero, 
\begin{equation} \label{I8}
0 < \underline{r} \leq r \leq \Ov{r}. 
\end{equation}
\item 
{\it Besov regularity.}
The density $r$, and the velocity $\vU$ belong to certain Besov spaces specified below.
\item 
{\it Positivity of the velocity gradient.}
There exists a function $D \in L^1(0,T)$ such that
\begin{equation} \label{I9}
\frac{1}{2}(\Grad \vU + \Grad \vU^t): (\xi \otimes \xi) + D |\xi|^2 \geq 0\ \mbox{for any}\ \xi \in R^N.  
\end{equation}

\end{itemize}

\begin{Remark} \label{IR1}
Note that \eqref{I9} in fact does not require positivity of the velocity gradient. It may be seen as a one sided Lipschitz condition
in the spirit of Bouchut and James \cite{BouJam2}, \cite{BouJam1}. As observed in \cite{FeKr2014}, the 1-D rarefaction waves satisfy 
\eqref{I9} with $D = 0$.
 
\end{Remark}

The assumption of Besov regularity seems quite natural in the context of compressible fluids, cf. Chen and Glimm \cite{CheGli}. Here, 
we are motivated by the proof of the one sided implication in Onsager's conjecture by 
Constantin, E, and Titi \cite{ConETit}, and its subsequent generalization to the compressible Euler system in 
\cite{FeGwGwWi}. In particular, we use the fact the Besov functions can be regularized by convolution kernels and the resulting commutators with non--linear superpositions can be effectively controlled.

The paper is organized as follows. In Section \ref{P}, we introduce the basic concepts and state our main result on uniqueness of weak solutions (see Theorem \ref{TP1}) Our strategy is based on regularizing the weak solution 
$r$, $\vU$ and using the resulting expressions as test functions in the relative energy inequality \eqref{I6}, see Section \ref{R}.
Next, we perform the limit in the regularized relative energy inequality. To this end, 
in Section \ref{C},
we prove certain commutator estimates that may be of independent interest. Finally, in Section \ref{D}, we extend the uniqueness result to the more general class of \emph{dissipative} solutions introduced recently in \cite{BreFeiHof19} (see Theorem \ref{DT1}). The paper is concluded by a short discussion how to accommodate solutions of the Riemann problem in Section \ref{Z}.

\section{Preliminaries, main results}
\label{P}  

We say that $\vr$, $\vm$ is an \emph{admissible weak solution} to the Euler system \eqref{I1}--\eqref{I3} 
in $(0,T) \times \Q$ if the following 
integral identities are satisfied:
\begin{equation} \label{P1}
\int_0^\tau \intO{ \left[ \vr \partial_t \varphi + \vm \cdot \Grad \varphi \right] } \dt = \intO{\vr(\tau, \cdot) \varphi(\tau, \cdot)}- \intO{\vr_0 \varphi(0, \cdot)}
\end{equation}
for any $0 \leq \tau < T$, and
any $\varphi \in C^1_c([0, T) \times \Q)$; 
\begin{equation} \label{P2}
\begin{split}
\int_0^\tau &\intO{ \left[ \vm \cdot \partial_t \bfphi + \frac{\vm \otimes \vm}{\vr}: \Grad \bfphi + p(\vr) 
\Div \bfphi \right] } \dt \\&= \intO{ \vm(\tau, \cdot) \cdot \bfphi(\tau, \cdot) } - \intO{ \vm_0 \cdot \bfphi(0, \cdot) }
\end{split}
\end{equation}
for any $0 \leq \tau < T$, and
any $\bfphi \in C^1_c([0, T) \times \Q; R^N)$;
\begin{equation} \label{P3}
\intO{ \left[ \frac{1}{2} \frac{ |\vm|^2 }{\vr} + P(\vr) \right](\tau, \cdot) } \leq
\intO{ \left[ \frac{1}{2} \frac{ |\vm|^2 }{\vr} + P(\vr) \right](s, \cdot) } \leq  
\intO{ \left[ \frac{1}{2} \frac{ |\vm_0|^2 }{\vr_0} + P(\vr_0) \right] }
\end{equation}
for any $0 \leq \tau < T$ and a.a. $0 < s < \tau$.

\subsection{Besov spaces} 

For $0 < \alpha < 1$, $1 \leq p < \infty$, 
and a domain $Q \subset \Ov{Q} \subset (0,T) \times \Q$, 
we introduce the Besov norm  
\begin{equation} \label{BN}
\| v \|_{B^{\alpha, \infty}_p (Q)} = \| v \|_{L^p(Q)} + \sup_{\eta \in R^{N+1},
\eta \ne 0, Q + \eta \subset (0,T) \times \Q} \frac{ \| v(\cdot + \eta) - v(\cdot) \|_{L^p(Q)} }{|\eta|^\alpha}.
\end{equation}

Let $[v]_\ep \equiv \eta_\ep * v$ denote the regularization via convolution with a family of regularizing kernels. 
The following estimates are well known, see e.g. Constantin et al. \cite{ConETit}:
\begin{equation} \label{P4}
\left\| [v]_\ep - v \right\|_{L^p(Q)} \leq \ep^\alpha \| v \|_{B^{\alpha, \infty}_p (Q)},
\end{equation}
\begin{equation} \label{P5}
\left\| \Grad [v]_\ep \right\|_{L^p(Q)} \leq \ep^{\alpha - 1} \| v \|_{B^{\alpha, \infty}_p (Q)}.
\end{equation}

\subsection{Main result}

Having collected all necessary preliminary material
we are ready to establish our main result concerning uniqueness in the class of \emph{weak} solutions. 

\begin{Theorem} \label{TP1}

Let $\vr$, $\vm$ be an admissible weak solution of the Euler system \eqref{I1}--\eqref{I3} in $[0,T) \times \Q$ with the initial data
\[
\vr_0 \geq 0,\ \intO{ \left[ \frac{1}{2} \frac{|\vm_0|^2}{\vr_0} + P(\vr_0) \right] } < \infty.
\]
Let $\tvr = r$, $\tilde \vm = r \vU$ be another weak solution of the same problem satisfying:
\begin{itemize}
\item
\begin{equation} \label{P6}
\begin{split}
r &\in B^{\alpha, \infty}_{p}((\delta,T) \times \Q)) \cap C([0,T]; L^1(\Q)),\\ 
\vU &\in B^{\alpha, \infty}_{p}((\delta,T) \times \Q; R^N)) \cap C([0,T]; L^1(\Q; R^N)),
\ \mbox{for any}\ \delta > 0,
\end{split}
\end{equation}
with
\[
\alpha > \frac{1}{2}, \ p \geq \frac{4 \gamma}{\gamma - 1};
\]
\item
\[
0< 
\underline{r} \leq r(t,x) \leq \Ov{r}, \ | \vU(t,x)| \leq \Ov{U}\ \mbox{for a.a.} \ (t,x) \in (0,T) \times \Q; 
\]

\item

there exists $D \in L^1(0,T)$ such that
\begin{equation} \label{P7}
\intO{ \left[ - \xi \cdot \vU(\tau, \cdot) (\xi \cdot \Grad) \varphi  + D(\tau) |\xi|^2 \varphi \right] } \geq 0 
\end{equation}
for any $\xi \in R^N$, and any $\varphi \in C(\Q)$, $\varphi \geq 0$.

\end{itemize}

Then
\[
\vr = \tvr,\ \vm = \tilde{\vm} \ \mbox{a.a. in}\ (0,T) \times \Q.
\]

\end{Theorem}

\begin{Remark} \label{PR1}

Strictly speaking, the weak solution $r$, $\vU$ need not satisfy the energy inequality \eqref{I5}. Sufficient conditions 
for \eqref{I5} to hold can be found in \cite{FeGwGwWi}.  

\end{Remark}

\begin{Remark} \label{PR2}

Condition \eqref{P6} should be understood in the sense that the Besov norm introduced in \eqref{BN} 
is bounded by a constant that may depend on $\delta > 0$.

\end{Remark}

The following two sections are devoted to the proof of Theorem \ref{TP1}. Generalizations to the class of dissipative solutions 
are obtained in Section \ref{D}.

\section{Relative energy inequality}
\label{R}

Let $\vr$, $\vm$ be an admissible weak solution to the Euler system \eqref{I1}--\eqref{I3}. We rewrite the relative energy inequality 
\eqref{I6} in a more convenient form:
\begin{equation} \label{R1}
\begin{split}
\int_{\Q} &\E \left( \vr, \vm \Big| r , \vc{U} \right) (\tau, \cdot) \dx \leq \intO{ \E \left( \vr, \vm \Big| r, 
\vc{U} \right) 
(s, \cdot) }\\
&- \int_s^\tau \intO{ \left[ \vr \left( \vU - \frac{\vm}{\vr} \right) \cdot \Grad \vU \cdot \left( \vU - \frac{\vm}{\vr} \right) 
+ \Big( p(\vr) - p'(r) (\vr - r) - p(r) \Big) \Div \vU \right]} \dt 
\\
&+ \int_s^\tau \intO{\Big[  \partial_t (r  \vU) + 
\Div\left( r \vU \otimes \vU  \right) + \Grad p(r) \Big] \cdot \frac{1}{r} \Big( \vr \vU - {\vm} \Big)  
 }\dt
\\
&+ \int_s^\tau \intO{ \Big[ \partial_t r + \Div (r \vU ) \Big] \left[ \left( 1 - \frac{\vr}{r}  \right) p'(r) + 
\frac{1}{r} \vU \cdot \left( \vm - \vr \vU \right) \right]  } \dt 
\end{split}
\end{equation} 
for any $0 \leq \tau < T$, a.a. $0 < s < \tau$, and any $r > 0$, $\vU$ continuously differentiable in $[s, T)$. Note 
that this is in fact a localized version of \eqref{I6} that can be obtained thanks to the appropriate form of the 
energy inequality \eqref{P3}. Moreover, inequality \eqref{R1} remains valid for $s=0$ as long as the relative energy at 
$s = 0$ is expressed in terms for the initial data.

Now, we fix $\tau > 0$, $0 < s < \tau$ and consider the space--time regularization $[r]_\ep$, $[\vU]_\ep$ as test functions in \eqref{R1}:
\begin{equation} \label{R2}
\begin{split}
\int_{\Q} &\E \left( \vr, \vm \Big| [r]_\ep , [\vc{U}]_\ep \right) (\tau, \cdot) \dx \leq \intO{ \E \left( \vr, \vm \Big| [r]_\ep, 
[\vc{U}]_\ep \right) 
(s, \cdot) }\\
&- \int_s^\tau \intO{ \vr \left( [\vU]_\ep - \frac{\vm}{\vr} \right) \cdot \Grad [\vU]_\ep \cdot \left( [\vU]_\ep - \frac{\vm}{\vr} \right) } \dt  \\ 
&- \int_s^\tau \intO{ \Big( p(\vr) - p'([r]_\ep) (\vr - [r]_\ep) - p([r]_\ep) \Big) \Div [\vU]_\ep } \dt 
\\
&+ \int_s^\tau \intO{\Big[  \partial_t ([r]_\ep  [\vU]_\ep) + 
\Div\left( [r]_\ep [\vU]_\ep \otimes [\vU]_\ep  \right) + \Grad p([r]_\ep) \Big] \cdot \frac{1}{[r]_\ep} \Big( \vr [\vU]_\ep - {\vm} \Big)  
 }\dt
\\
&+ \int_s^\tau \intO{ \Big[ \partial_t [r]_\ep + \Div ([r]_\ep [\vU]_\ep ) \Big] \left[ \left( 1 - \frac{\vr}{[r]_\ep}  \right) p'([r]_\ep) + 
\frac{1}{[r]_\ep} [\vU]_\ep \cdot \left( \vm - \vr [\vU]_\ep \right) \right]  } \dt. 
\end{split}
\end{equation} 

It follows from \eqref{P7} that 
\begin{equation} \label{R3}
\xi \cdot \Grad [\vU]_\ep \cdot \xi + |\xi|^2 [D]_\ep  \geq 0 \ \mbox{for any}\ \xi \in R^N. 
\end{equation}
Consequently, thanks to the isentropic pressure, relation \eqref{R2} reduces to
\begin{equation} \label{R4}
\begin{split}
\int_{\Q} &\E \left( \vr, \vm \Big| [r]_\ep , [\vc{U}]_\ep \right) (\tau, \cdot) \dx \leq \intO{ \E \left( \vr, \vm \Big| [r]_\ep, 
[\vc{U}]_\ep \right) 
(s, \cdot) } + \int_s^\tau D \E \left( \vr, \vm \Big| [r]_\ep , [\vc{U}]_\ep \right) \dt \\
&+ \int_s^\tau \intO{\Big[  \partial_t ([r]_\ep  [\vU]_\ep) + 
\Div\left( [r]_\ep [\vU]_\ep \otimes [\vU]_\ep  \right) + \Grad p([r]_\ep) \Big] \cdot \frac{1}{[r]_\ep} \Big( \vr [\vU]_\ep - {\vm} \Big)  
 }\dt
\\
&+ \int_s^\tau \intO{ \Big[ \partial_t [r]_\ep + \Div ([r]_\ep [\vU]_\ep ) \Big] \left[ \left( 1 - \frac{\vr}{[r]_\ep}  \right) p'([r]_\ep) + 
\frac{1}{[r]_\ep} [\vU]_\ep \cdot \left( \vm - \vr [\vU]_\ep \right) \right]  } \dt. 
\end{split}
\end{equation} 

Next, as $r$ and $\vU$ are weak solutions to the Euler system, we get 
\[
\partial_t [r]_\ep = 
- \left[ \Div (r \vU) \right]_\ep,\ \partial_t [ r \vU]_\ep = 
- \left[ \Div (r \vU \otimes \vU) + \Grad p(r) \right]_\ep .
\]
Thus, finally, 
\begin{equation} \label{R5}
\begin{split}
&\int_{\Q} \E \left( \vr, \vm \Big| [r]_\ep , [\vc{U}]_\ep \right) (\tau, \cdot) \dx \leq \intO{ \E \left( \vr, \vm \Big| [r]_\ep, 
[\vc{U}]_\ep \right) 
(s, \cdot) } + \int_s^\tau D \E \left( \vr, \vm \Big| [r]_\ep , [\vc{U}]_\ep \right) \dt \\
&+ \int_s^\tau \intO{\Big[   
\Div\left( [r]_\ep [\vU]_\ep \otimes [\vU]_\ep  \right) + \Grad p([r]_\ep)
- \left[ \Div (r \vU \otimes \vU) + \Grad p(r) \right]_\ep
 \Big] \cdot \frac{1}{[r]_\ep} \Big( \vr [\vU]_\ep - {\vm} \Big)  
 }\dt
\\
&+ \int_s^\tau \intO{\Big[   
\partial_t ([r]_\ep [\vU]_\ep) - \partial_t [r \vU ]_\ep 
\Big] \cdot \frac{1}{[r]_\ep} \Big( \vr [\vU]_\ep - {\vm} \Big)  
 }\dt
\\
&+ \int_s^\tau \intO{ \Big[ \Div ([r]_\ep [\vU]_\ep )  - \left[ \Div (r \vU) \right]_\ep
\Big] \left[ \left( 1 - \frac{\vr}{[r]_\ep}  \right) p'([r]_\ep) + 
\frac{1}{[r]_\ep} [\vU]_\ep \cdot \left( \vm - \vr [\vU]_\ep \right) \right]  } \dt. 
\end{split}
\end{equation}

In order to control the integrals on the right--hand  side of \eqref{R5}, we need the commutator estimates proved in the following section.

\section{Commutator estimates}
\label{C}

The following results is necessary to control the integrals on the right--hand side of the inequality \eqref{R5}.

\begin{Lemma} \label{LC1}

Let $Q$ be a bounded domain in $R^M$. Suppose that $\mathbb{V}: \tilde Q \to R^k$ belongs to the Besov space 
$B^{\alpha, \infty}_p(Q, R^k)$, $p \geq 2$, where $\tilde Q \subset R^M$ is another domain containing $\Ov{Q}$ in its interior.
Let $\eta^\ep$ be a standard family of regularizing kernels, ${\rm supp}[ \eta^{\ep} ] \subset \{ |y| < \ep \}$. Let 
$G : K \to R$ be a twice continuously differentiable function defined on an open set $K \subset R^k$ containing the 
closure of the range of $\mathbb{V}$.

Then 
\[
\left\| \nabla_y G( [\mathbb{V}]_\ep ) - \nabla_y [ G(\mathbb{V}) ]_\ep \right\|_{L^{\frac{p}{2}} (Q; R^M) }
\leq \ep^{2 \alpha - 1} c(\| G \|_{C^2(K)}) \left( 1 + \left\| \mathbb{V} \right\|^2_{B^{\alpha, \infty}_p (Q; R^k)} \right)
\] 
for $\nabla_y = (\partial_{y_1}, \dots, \partial_{y_M})$.

\end{Lemma}

\begin{proof}

On one hand, 
\[
\begin{split}
\Big( \nabla_y G( [\mathbb{V}]_\ep ) - \nabla_y [ G(\mathbb{V}) ]_\ep \Big) 
= \Big( G' ( [\mathbb{V}]_\ep ) - G' (\mathbb{V}) \Big) \nabla_y [\mathbb{V}]_\ep + 
G' (\mathbb{V}) \nabla_y [\mathbb{V}]_\ep - \nabla_y [ G(\mathbb{V}) ]_\ep,
\end{split}
\]
where, by virtue of \eqref{P4}, \eqref{P5}, and H\" older's inequality, 
\[
\left\|
\Big( G' ( [\mathbb{V}]_\ep ) - G' (\mathbb{V}) \Big) \nabla_y [\mathbb{V}]_\ep \right\|_{L^{\frac{p}{2}}(Q; R^M)} \leq \ep^{2 \alpha - 1} 
c(\| G \|_{C^2(K)}) \left\| \mathbb{V} \right\|_{B^{\alpha, \infty}_{p} (Q; R^k) }.
\]

On the other hand, 
\[
\begin{split}
G' (\mathbb{V}) \nabla_y [\mathbb{V}]_\ep(y) &- \nabla_y [ G(\mathbb{V}) ]_\ep(y) = 
G'(\mathbb{V}(y)) [\nabla_y \eta^\ep * \mathbb{V}] (y) - 
[\nabla_y \eta^\ep * G(\mathbb{V}) ] (y) \\
&= G'(\mathbb{V}(y)) \int_{R^M} \nabla_y \eta^\ep (y- z) \mathbb{V}(z) {\rm d}z - 
\int_{R^M} \nabla_y \eta^\ep (y- z) G(\mathbb{V}(z)) {\rm d}z \\
&= - \int_{R^M} \nabla_y \eta^\ep (y - z) \Big[ G(\mathbb{V}(z)) + G'(\mathbb{V}(y))(\mathbb{V}(y) - \mathbb{V}(z)) 
- G(\mathbb{V}(y)) \Big] {\rm d}z 
\\
&=  \int_{R^M \setminus \{ 0 \} } |h|^{2\alpha} \nabla_y \eta^\ep (h) \frac{1}{|h|^{2\alpha}}{\Big[ G(\mathbb{V}(y-h)) + G'(\mathbb{V}(y))(\mathbb{V}(y) - \mathbb{V}(y-h)) 
- G(\mathbb{V}(y)) \Big]}\ {\rm d}h\\ 
&\leq c(\| G \|_{C^2(K)})\int_{R^M \setminus \{0\}} \frac{ |\mathbb{V} ( y - h) - \mathbb{V} (y) |^2 }{|h|^{2 \alpha}}  |h |^{2 \alpha - 1}|h||\nabla_y \eta^\ep (h)| \ {\rm d}h.   
\end{split}
\]
Consequently, by virtue of Jensen's and H\" older's inequalities,  
\[
\begin{split}
&\left\| G' (\mathbb{V}) \nabla_y [\mathbb{V}]_\ep(y) - \nabla_y [ G(\mathbb{V}) ]_\ep(y) \right\|_{L^{\frac{p}{2}}(Q; R^M)}\\ 
&\leq \ep^{2\alpha - 1} c(\| G \|_{C^2(K)}) \sup_{0< |h| \leq \ep} \left\| \frac{ \mathbb{V}(\cdot + h) - \mathbb{V}(\cdot) }{|h|^\alpha}
\right\|_{L^{{p}}(Q)}^2 \leq \ep^{2 \alpha - 1} c(\| G \|_{C^2(K)})  \left\| \mathbb{V} \right\|^2_{B^{\alpha, \infty}_p (Q; R^k)}. 
\end{split}
\]

\end{proof}

Under the hypothesis \eqref{P6}, we can perform the limit for $\ep \to 0$ in the inequality \eqref{R5}. 
Indeed, as $\vr$, $\vm$ satisfies the energy inequality, we have 
\begin{equation} \label{cont}
\vr \in C_{{\rm weak}}([0, T]; L^\gamma(\Q)), \ \vm \in C_{{\rm weak}}([0,T]; L^{\frac{2 \gamma}{\gamma+ 1}}(\Q; R^N)).
\end{equation}
Consequently, we infer that
\begin{equation} \label{C1}
\int_{\Q} \E \left( \vr, \vm \Big| r , \vc{U} \right) (\tau, \cdot) \dx \leq \intO{ \E \left( \vr, \vm \Big| r, 
\vc{U} \right) 
(s, \cdot) } + \int_s^\tau D \E \left( \vr, \vm \Big| r , \vc{U} \right) \dt 
\end{equation} 
for a.a. $0 < s < \tau$.

Thus the last step in the proof of Theorem \ref{TP1} is to let $s \to 0$. In view of the energy inequality
\eqref{P3} and the fact that $r$ and $\vU$ are uniformly bounded, it is enough to observe that 
\begin{equation} \label{cont1}
r \in C([0,T]; L^q(\Q)),\ \vU \in C([0,T]; L^q(\Q; R^N))\ \mbox{for any}\ 1 \leq q < \infty
\end{equation}
in accordance with hypothesis \eqref{P6}. Now, we choose a suitable sequence $s_n \searrow 0$ such that 
\[
\begin{split}
\int_{\Q} & \E \left(\vr, \vm \Big| r, \vU \right) (s_n, \cdot) \dx = 
\intO{ \left[ \frac{1}{2} \vr \left| \frac{\vm}{\vr} - \vU \right|^2 + P(\vr) - P'(r) (\vr - r) - P(r) \right]
(s_n, \cdot) } \\
&= \intO{ \left[ \frac{1}{2} \frac{|\vm|^2}{\vr} + P(\vr) \right](s_n, \cdot) } - \intO{ \vm \cdot \vU (s_n, \cdot) }
+ \intO{ \frac{1}{2} \vr |\vU|^2 (s_n, \cdot) }\\
&- \intO{ \left[ P'(r) (\vr - r) + P(r) \right] (s_n, \cdot)}\\ 
&\leq \intO{ \left[ \frac{1}{2} \frac{|\vm_0|^2}{\vr_0} + P(\vr_0) \right]} - \intO{ \vm \cdot \vU (s_n, \cdot) }
+ \intO{ \frac{1}{2} \vr |\vU|^2 (s_n, \cdot) }\\
&- \intO{ \left[ P'(r) (\vr - r) + P(r) \right] (s_n, \cdot)} \to 0 \ \mbox{as}\ s_n \to 0,
\end{split}
\] 
where we have used the continuity properties \eqref{cont}, \eqref{cont1} and the fact that the initial data coincide.

The proof of Theorem \ref{TP1} is complete.

\section{Dissipative solutions}
\label{D}

We conclude the paper by extending the previous uniqueness result to a large class of the \emph{dissipative solutions} to the 
Euler system. The leading idea, similar to that one proposed by Lions \cite{LI} in the context of incompressible fluids, is to consider 
the relative energy inequality \eqref{I6} as a proper definition of a class of solutions that goes even beyond the concept of weak solutions. It turns out, however, that a proper formulation is not via \eqref{I6} but rather in terms of the \emph{measure--valued} solutions
or rather their expected values. This approach has been developed in \cite{BreFeiHof19}, where it was shown that it is possible to select an appropriate dissipative solution in such a way that the resulting solution operator gives rise to a semigroup. Here, we only recall the main ideas referring the interested reader to \cite{BreFeiHof19} for details. 

\subsection{Dissipative measure valued solution}

\emph{Dissipative measure valued solution} to the Euler problem \eqref{I1}--\eqref{I3} consists of the following quantities:
 
\begin{itemize}
\item a parametrized system of probability measures (Young measure) 
\[
\mathcal{V} = \mathcal{V}_{t,x}: 
(t,x) \mapsto \mathcal{P}([0, \infty) \times R^N),\ t \geq 0,\ x \in \Q 
\]
where the symbol $\mathcal{P}$ denotes the space of Borel probability measures, 
\[
\mathcal{V} \in L^{\infty}_{{\rm weak -(*)}} ((0, \infty) \times \Q; \mathcal{P}([0, \infty) \times R^N));
\]

\item the energy concentration defect measures, 
\[
\mathfrak{C}_{\rm kin} \in L^\infty_{{\rm weak - (*)}} (0, \infty; \mathcal{M}^+(\Q)),
\ \mathfrak{C}_{\rm int} \in L^\infty_{{\rm weak - (*)}} (0, \infty; \mathcal{M}^+(\Q)) ;
\]

\item the convection and pressure concentration defect measures,
\[
\mathfrak{C}_{\rm conv} \in L^\infty_{{\rm weak - (*)}} (0, \infty; \mathcal{M}^+(S^{N-1} \times \Q )), 
\ \mathfrak{C}_{\rm press} \in L^\infty_{{\rm weak - (*)}} (0, \infty; \mathcal{M}^+(\Q)).
\]

\end{itemize}

The concentration measures are interrelated reflecting the consitutive assumptions:
\begin{equation} \label{D1}
\mathfrak{C}_{{\rm press}} = (\gamma - 1) \mathfrak{C}_{\rm int},\ \frac{1}{2} \int_{S^{N-1}} \mathfrak{C}_{\rm conv} 
\ {\rm d}S = \mathfrak{C}_{\rm kin}.
\end{equation}

Now, the total energy $\mathbb{E}$ is supposed to be non--increasing function of $t \in [0, \infty)$, 
\begin{equation} \label{D2}
\mathbb{E}(t) = \intO{ \left< \mathcal{V}_{t,x}; \frac{1}{2} \frac{ |\tilde \vm|^2}{ \tilde \vr} 
+ P(\tilde \vr) \right> } + \int_{\Ov{\Omega}} \D \mathfrak{C}_{\rm kin}(t) + \int_{\Ov{\Omega}} \D \mathfrak{C}_{\rm int}(t)  
\end{equation}
for a.a. $t > 0$.

The equation of continuity holds as  
\begin{equation} \label{D3}
\int_0^\tau \intO{ \Big[ \left< \mathcal{V}_{t,x} ; \tilde \vr \right> \partial_t \varphi + 
\left< \mathcal{V}_{t,x}; \tilde \vm \right>  \cdot \Grad \varphi \Big] } \dt = \intO{ \left< \mathcal{V}_{\tau,x}; \tvr \right> 
\varphi (\tau, \cdot)} - \intO{\vr_0 \varphi(0, \cdot)}
\end{equation}
for any $\varphi \in C^1_c([0, T) \times \Q)$. 

Finally, the momentum equation reads:
\begin{equation} \label{D4}
\begin{split}
\int_0^\tau &\intO{ \left[ \left< \mathcal{V}_{t,x}; \tilde \vm \right> \cdot \partial_t \bfphi + 
\left< \mathcal{V}_{t,x}; \frac{ \tilde \vm \otimes \tilde \vm }{\tvr} \right>: \Grad \bfphi + 
\left< \mathcal{V}_{t,x}; p(\tvr) \right> \Div \bfphi \right] } \dt 
\\
&+ \int_0^\tau \left( \int_{S^{N-1}} \int_\Omega \Grad \bfphi : (\xi \otimes \xi)   \D \mathfrak{C}_{\rm conv}(t, \xi,x) \right) \dt +\int_0^\tau \left( \int_{\Q} \Div \bfphi \ \D \mathfrak{C}_{{\rm press}}(t,x) \right) \dt
\\
&= \intO{ \left< \mathcal{V}_{\tau,x}; \tilde \vm \right> \cdot 
\bfphi (\tau, \cdot)} - \intO{\vm_0 \cdot \bfphi(0, \cdot)}
\end{split}
\end{equation}
for any $\bfphi \in C^1_c([0, T) \times \Q; R^N)$.

Following \cite{BreFeiHof19} we define the \emph{dissipative solutions} to the Euler system as follows:

\begin{Definition}[Dissipative solution] \label{DD1}

A \emph{dissipative solution} to the Euler system \eqref{I1}--\eqref{I3}, and the initial energy $E_0$ is a trio 
\[
\begin{split}
\vr \in C_{{\rm weak, loc}}([0, \infty); L^\gamma (\Q)),\ 
&\vm \in C_{{\rm weak, loc}}([0, \infty); L^\Gamma (\Q; R^N)),\\ 
&E: [0, \infty) \to [0, \infty)
\ \mbox{non--decreasing}
\end{split}
\]
such that there exists a dissipative measure--valued solution in the sense specified above so that
\[
\vr(t,x) = \left< \mathcal{V}_{t,x}; \tvr \right>,\ 
\vm(t,x) = \left< \mathcal{V}_{t,x}; \tilde \vm \right>, \mbox{with energy}\ E, \ E(0-) \equiv E_0.
\]

\end{Definition}

\subsection{Relative energy inequality for dissipative solutions}

The relative energy inequality \eqref{R1} can be generalized to the class of dissipative solutions. First we define, formally, 
the relative energy:
\[
\intO{ \mathcal{E} \left( \vr, \vm, E \Big| r, \vU \right) } \equiv E - \intO{ \vm \cdot \vU } 
+ \frac{1}{2} \intO{ \vr |\vU|^2 } - \intO{ P'(r) \vr } + \intO{ p(r) }.
\]
In accordance with Definition \eqref{DD1}, the relative entropy is a non--decreasing function of $t$, having well defined right and left 
limits at any $t \in [0, \infty)$, with the convention $E(0-) = E_0$. 

Using \eqref{D1}--\eqref{D4} we obtain, after a bit tedious but straightforward manipulation,
a variant of the relative energy inequality \eqref{I6}: 
\begin{equation} \label{D5}
\begin{split}
\int_{\Q} &\mathcal{E} \left( \vr, \vm, E \Big| r, \vU \right)(\tau+) \ \dx
\leq  \intO{ \mathcal{E} \left( \vr, \vm, E \Big| r, \vU \right)(s-)  }
\\
&+ \int_s^\tau \intO{\left[ (\vr \vU - \vm) \partial_t \vU + 
\Grad \vU : \left< \mathcal{V}_{t,x}; \tilde \vm \otimes \Big( \vU - \frac{\tilde \vm}{\tilde \vr} \Big) \right> +
 \Big( (p(r) - \left< \mathcal{V}_{t,x}; p(\tvr) \right> \Big) \Div \vU  \right] }\dt
\\
&+\int_s^\tau \intO{ \left[ (r - \vr) \partial_t P'(r) + (r \vc{U} - \vm ) \cdot \Grad P'(r) \right] }\dt\\
&-  \int_s^\tau \left( \int_{S^{N-1}} \int_\Omega \Grad \vU : (\xi \otimes \xi)   \D \mathfrak{C}_{\rm conv}(t, \xi,x) \right) \dt
 - (\gamma - 1) \int_s^\tau \left( \int_{\Q} \Div \vU \ \D \mathfrak{C}_{{\rm int}}(t,x) \right) \dt 
\end{split}
\end{equation} 
for any continuously differentiable $r > 0$, $\vU$, and
for any $\tau > 0$, $0 \leq s \leq \tau$. 

Finally, exactly as in Section \ref{R}, we may rewrite \eqref{D5} in the form 
\begin{equation} \label{D6}
\begin{split}
\int_{\Q} &\mathcal{E} \left( \vr, \vm, E \Big| r, \vU \right)(\tau+) \ \dx
\leq  \intO{ \mathcal{E} \left( \vr, \vm, E \Big| r, \vU \right)(s-)  }\\
&- \int_s^\tau \intO{ \Grad \vU : \left< \mathcal{V}_{t,x} ; \tvr \left( \vU - \frac{\tilde \vm}{\tvr} \right) \otimes \left( \vU - \frac{\tilde \vm}{\tvr} \right) \right> 
} \dt \\
- &\int_s^\tau \intO{ \Big( \left< \mathcal{V}_{t,x}; p(\tvr) \right> - p'(r) (\vr - r) - p(r) \Big) \Div \vU } \dt 
\\
&+ \int_s^\tau \intO{\Big[  \partial_t (r  \vU) + 
\Div\left( r \vU \otimes \vU  \right) + \Grad p(r) \Big] \cdot \frac{1}{r} \Big( \vr \vU - {\vm} \Big)  
 }\dt
\\
&+ \int_s^\tau \intO{ \Big[ \partial_t r + \Div (r \vU ) \Big] \left[ \left( 1 - \frac{\vr}{r}  \right) p'(r) + 
\frac{1}{r} \vU \cdot \left( \vm - \vr \vU \right) \right]  } \dt 
\\
&-  \int_s^\tau \left( \int_{S^{N-1}} \int_\Omega \Grad \vU : (\xi \otimes \xi)   \D \mathfrak{C}_{\rm conv}(t, \xi,x) \right) \dt
 - (\gamma - 1) \int_s^\tau \left( \int_{\Q} \Div \vU \ \D \mathfrak{C}_{{\rm int}}(t,x) \right) \dt 
\end{split}
\end{equation} 
for any continuously differentiable $r > 0$, $\vU$, and
for any $\tau > 0$, $0 \leq s \leq \tau$.

\subsection{Uniqueness in the class of dissipative solutions}

With the relative energy inequality \eqref{D6} at hand, we may use the relations \eqref{D1}, \eqref{D2} between the concentration defect 
measures and follow step by step the arguments of Sections \ref{R}, \ref{C} to extend Theorem \ref{TP1} to the class of dissipative solutions: 
\begin{Theorem} \label{DT1}

Let $[\vr, \vm, E]$ be a dissipative solution of the Euler system \eqref{I1}--\eqref{I3}, with the initia data 
\[
\vr_0 \geq 0,\ E_0 = \intO{ \left[ \frac{1}{2} \frac{|\vm_0|^2}{\vr_0} + P(\vr_0) \right] } < \infty.
\]
Let $\tvr = r$, $\tilde \vm = r \vU$ be a weak solution of the same problem, with the initial energy
\[
\intO{ \left[ \frac{1}{2} r_0^2 |\vU_0|^2 + P(r_0) \right] } = E_0,
\]
satisfying:
\begin{itemize}
\item
\[
\begin{split}
r &\in B^{\alpha, \infty}_{p}((\delta,T) \times \Q)) \cap C([0,T]; L^1(\Q)),\\ 
\vU &\in B^{\alpha, \infty}_{p}((\delta,T) \times \Q; R^N)) \cap C([0,T]; L^1(\Q; R^N)),
\ \mbox{for any}\ \delta > 0,
\end{split}
\]
with
\[
\alpha > \frac{1}{2}, \ p \geq \frac{4 \gamma}{\gamma - 1};
\]
\item
\[
0< 
\underline{r} \leq r(t,x) \leq \Ov{r}, \ | \vU(t,x)| \leq \Ov{U}\ \mbox{for a.a.} \ (t,x) \in (0,T) \times \Q; 
\]
\item
there exists $D \in L^1(0,T)$ such that
\[
\intO{ \left[ - \xi \cdot \vU(\tau, \cdot) (\xi \cdot \Grad) \varphi  + D(\tau) |\xi|^2 \varphi \right] } \geq 0 
\]
for any $\xi \in R^N$, and any $\varphi \in C(\Q)$, $\varphi \geq 0$.

\end{itemize}

Then
\[
\vr = \tvr,\ \vm = \tilde{\vm} \ \mbox{a.a. in}\ (0,T) \times \Q.
\]

\end{Theorem}

\section{Riemann problem and non--periodic boundary conditions}
\label{Z} 

We conclude by a short discussion how to accommodate different kind of boundary conditions. First easy observation is that 
the above results are valid for an arbitrary flat torus, 
\[
\Omega = \Pi_{i = 1}^N [ - L_i, L_i ]|_{\{ -L_i,L_i \} }.
\]

Now, consider the situation 
\[
\Omega = R \times \left( [-1,1]|_{-1,1} \right)^2,
\]
where the ``regular solution'' $r$, $\vU$, depends only on the variable $x_1$ and the velocity has only one non--zero component, 
\[
r(t, x_1, x_2, x_3) = r(t,x_1), \ \vU(t, x_1, x_2, x_3) = [U(t,x_1),0,0],\ r, \ U \ \mbox{extended to be constant in}\ (x_2, x_3).
\]
Note that the one--sided Lipschitz condition \eqref{P7} reduces to 
\[
U(t, x_1) + D(t) x_1 \ \mbox{non--decreasing for}\ x_1 \in R, 
\]
which includes, in particular, the rarefaction waves studied in \cite{CheChe} and \cite{FeKr2014}. Motivated \cite{CheChe}, we impose 
Riemann like far field conditions, 
\[
r(t,x_1) = r_L > 0,\ U(t, x_1) = U_L \ \mbox{for}\ x_1 < - M,\ 
r(t,x_1) = r_R > 0,\ U(t, x_1) = U_R \ \mbox{for}\ x_1 > M.
\]

Given $T > 0$, our goal is to find an extended solution $\tilde r$, $\tilde U$ satisfying 
\[
\begin{split}
\tilde r(t,x_1) &= r(t, x_1) , \ \tilde U (t,x_1) = U (t,x_1) \ \mbox{whenever}\ t \in [0,T],\ x_1 \leq M,\\ 
\tilde r(t, x_1) &= r_L,\ U(t,x_1) = U_L \ \mbox{if}\ t \in [0,T], \ x_1 > K 
\ \mbox{for some}\ K > M, \\
(r, U) &\in C^1([0,T] \times [M, \infty).
\end{split}
\]
Then the problem can be studied on the torus $[-K,K]|_{\{ -K, K \}} \times \left( [-1,1]_{\{-1, 1\}} \right)^2$.

Making obvious space shift transformation and taking the finite speed of propagation into account, the task reduces to finding 
a $C^1$ solution $\hat r$, $\hat U$ such that 
\begin{equation} \label{Z1}
\hat r > 0, \ \hat r(t, x_1) = r_R,\ \hat U(t, x_1) = U_R,\ x_1 < 0,\ 
\hat r(t,x_1) = r_L,\ \hat U(t, x_1) = U_L \ \mbox{for}\ x_1 > K.
\end{equation}
To this end we first choose the initial data $r_0$, $U_0$ with uniformly bounded $C^1(R)$ norm satisfying \eqref{Z1} (with $t=0$). 
By virtue of the known results on solvability of symmetric hyperbolic systems, see e.g. Li and Yu \cite{LiYu}, there exists 
a $C^1$ solution $r^\ep$, $U^\ep$ defined on a possibly short time interval $[0,\ep]$. Thus rescaling 
\[
\hat r(t, x_1) = r^\ep \left( t \frac{\ep}{T}, x_1 \frac{\ep}{T} \right), 
\hat U(t, x_1) = U^\ep \left( t \frac{\ep}{T}, x_1 \frac{\ep}{T} \right),\ t \in [0,T], \ x \in R
\]
yields the desired result. 

\bigskip

\centerline{\bf Acknowledgement}

\medskip

The paper was written during the stay of E.F. at the Centre for Applicable Mathematics, Tata Institute of Fundamental Research, 
Bangalore which support and hospitality are gladly acknowledged.

\def\cprime{$'$} \def\ocirc#1{\ifmmode\setbox0=\hbox{$#1$}\dimen0=\ht0
  \advance\dimen0 by1pt\rlap{\hbox to\wd0{\hss\raise\dimen0
  \hbox{\hskip.2em$\scriptscriptstyle\circ$}\hss}}#1\else {\accent"17 #1}\fi}

%\bibliography{citace}
%\bibliographystyle{plain}

\end{document}